\documentclass[11pt]{amsart}
\title[Mixed Hodge structures and Sullivan's minimal models of Sasakian manifolds]
{Mixed Hodge structures and Sullivan's minimal models of Sasakian manifolds}

\author{Hisashi Kasuya}

\usepackage{amssymb}
\usepackage{amsmath}
\usepackage{amscd}
\usepackage{amstext}
\usepackage{amsfonts}
\usepackage[all]{xy}

\theoremstyle{plain}

\theoremstyle{plain}

\theoremstyle{plain}

\theoremstyle{plain}
\newtheorem{theorem}{Theorem}[section] 
\theoremstyle{remark}
\newtheorem{remark}[theorem]{Remark}
\theoremstyle{Main result}
\newtheorem{main result}{Main result}
\theoremstyle{lemma}

\theoremstyle{definition}
\newtheorem{definition}[theorem]{Definition}
\theoremstyle{proposition}
\newtheorem{proposition}[theorem]{Proposition}
\theoremstyle{corollary}
\newtheorem{corollary}[theorem]{Corollary}
\theoremstyle{remark}
\newtheorem{example}[theorem]{Example}
\theoremstyle{remark}

\theoremstyle{remark}

\theoremstyle{assumption}

\address[Hisashi Kasuya]{Department of Mathematics, Tokyo Institute of Technology, 2-12-1 Ookayama, Meguro-ku, Tokyo 152-8551, JAPAN}
\email{kasuya@math.titech.ac.jp}


\newcommand{\C}{\mathbb{C}}
\newcommand{\R}{\mathbb{R}}

\newcommand{\g}{\frak{g}}
\newcommand{\n}{\frak{n}}

\begin{document} 

\maketitle
\begin{abstract}
We show that the Malcev Lie  algebra of the fundamental group of a compact $2n+1$-dimensional Sasakian manifold with $n\ge 2$ admits a quadratic presentation by using Morgan's
bigradings of minimal models of mixed-Hodge diagrams.
By using bigradings of minimal models, we also simplify the proof of the result of Cappelletti-Montano, De Nicola, Marrero and Yudin on Sasakian nilmanifolds.
\end{abstract}
\section{Introduction}
Let $\Gamma$ be a group and $\Gamma=\Gamma_{1}\supset \Gamma_{2}\supset \Gamma_{3} \dots$
the lower central series (i.e. $\Gamma_{i}=[\Gamma_{i-1},\Gamma]$).
Consider the tower of nilpotent groups 
\[\dots\to \Gamma/\Gamma_{3}\to \Gamma/\Gamma_{2}\to \{e\}.
\]
Then it is possible to "tensor" these nilpotent groups with $\R$ or $\C$ (\cite{Mal}, \cite{DGMS}, \cite{ABC}) and we obtain the tower of real  nilpotent Lie algebras
 \[\dots\to {\frak n}_{3}\to {\frak n
}_{2}\to \{ 0\}.
\]
The inverse limit of this tower is called the Malcev Lie algebra of $\Gamma$. 

Let $M$ be a manifold.
By Sullivan's de Rham homotopy theory (\cite{Sul}),  the Malcev Lie algebra of the fundamental group $\pi_{1}(M)$ can be studied by the  differential forms on $M$.
The formality of compact K\"ahler manifolds (\cite{DGMS}) implies that if $M$ is a compact K\"ahler manifold, then the Malcev Lie algebra of the fundamental group $\pi_{1}(M)$ admits a quadratic presentation (i.e. is a quotient of a free Lie algebra by an ideal generated in degree two).
This fact is very useful to study whether a given finitely generated group can be the fundamental group of a compact K\"ahler manifold.

In this paper, we consider Sasakian manifolds.
Sasakian manifolds constitute an odd-dimensional counterpart of the class of K\"ahler manifolds.
We are interested in whether  the Malcev Lie algebras of the fundamental groups of compact Sasakian manifolds admit quadratic presentations.

First, we see important examples.
We consider the $2n+1$-dimensional real  Heisenberg group $H_{2n+1}$ which is the group of matrices of the form
\[\left(
\begin{array}{ccc}
1& x&z  \\
0&    I &\,^{t} y\\
0&0&1  
\end{array}
\right)
\]
where $I$ is the $n\times n$ unit matrix, $x,y \in \R^{n}$ and $z\in \R$. 
For any  lattice $\Gamma$ in $H_{2n+1}$, the nilmanifold $\Gamma \backslash H_{2n+1}$ admits a Sasakian structure and $\pi_{1}(\Gamma \backslash H_{2n+1})\cong \Gamma$.
It is known that for any $n\ge 2$,  the Malcev Lie algebra of $\Gamma$  admits a quadratic presentation (see \cite{CT}). 
However, if $n=1$, then the Malcev Lie algebra of $\Gamma$ can not  admit a quadratic presentation. 

In this paper, we extend this observation.
We prove the following theorem.

\begin{theorem}\label{int}
Let $M$ be a compact $2n+1$-dimensional Sasakian manifold with $n\ge 2$.
Then the Malcev Lie algebra of $\pi_{1}(M)$ admits a quadratic presentation.
\end{theorem}
\begin{remark}
It seems that this fact is known by some experts.
However, it is difficult to find an explicit proof in some reference.
\end{remark}

We notice that Sasakian manifolds are not formal in general, unlike K\"ahler manifolds.
See \cite{Bis}, \cite{Mu} for the formality of sasakian manifolds.
But we can apply some Hodge theoretical  properties of Sasakian manifolds like algebraic varieties.
By using Morgan's techniques  of mixed Hodge diagrams \cite{Mor},
we will show the following theorem.

\begin{theorem}\label{SasMMM11int}
Let $M$ be a $2n+1$-dimensional compact Sasakian manifold and $A^{\ast}_{\C}(M)$ the de Rham complex of $M$.
Consider the minimal model $\mathcal M$ (resp $1$-minimal model) of $A^{\ast}_{\C}(M)$ with a quasi-isomorphism  (resp. $1$-quasi-isomorphism) $\phi:\mathcal M\to A_{\C}^{\ast}(M)$.
Then we have:
\begin{enumerate}
\item The real de Rham cohomology $H^{\ast}(M,\R)$ admits a 
$\R$-mixed-Hodge structure.

\item 
$\mathcal M^{\ast}$ admits a bigrading 
\[\mathcal M^{\ast}=\bigoplus_{p,q\ge0}\mathcal M^{\ast}_{p,q}
\]
such that $\mathcal M^{\ast}_{0,0}=\mathcal M^{0}=\C$ and 
the product and the differential are of type $(0,0)$.

\item
Consider the bigrading $H^{\ast}(M,\C)=\bigoplus V_{p,q}$ for the  $\R$-mixed-Hodge structure.
Then the induced map $\phi^{\ast}:H^{\ast}(\mathcal M^{\ast})\to H^{\ast}(M,\C)$ 
sends $H^{\ast}(\mathcal M^{\ast}_{p,q})$ to $V_{p,q}$.

\end{enumerate}
\end{theorem}

By using this theorem, we prove Theorem \ref{int}.
Moreover, we will give another application of Theorem \ref{SasMMM11int}.
In  \cite{Nil}, it is proved that 
a compact $2n+1$-dimensional nilmanifold admits a Sasakian structure if and only if it is a Heisenberg nilmanifold $H_{2n+1}/\Gamma$.
By using Theorem \ref{SasMMM11int}, we can simplify the proof of this result (Section \ref{nini}).

 {\bf  Acknowledgements.} 

The author  would  like to thank   Sorin Dragomir for  helpful comments.
This research is supported by  JSPS Research Fellowships for Young Scientists.

\section{Mixed Hodge structures}

\begin{definition}
An {\em$\R$-Hodge structure} of weight $n$ on a $\R$-vector space $V$ is a finite bigrading 
\[V_{\C}=\bigoplus_{p+q=n}V_{p,q}
\]
on the complexification $V_{\C}=V\otimes \C$
such that
\[\overline{V_{p,q}}=V_{q,p},
\]
or equivalently, a finite decreasing filtration $F^{\ast}$ on $V_{\C}$ 
such that 
\[F^{p}(V_{\C})\oplus \overline {F^{n+1-p}(V_{\C})}
\]
for each $p$. 

\end{definition}

\begin{definition}
An {\em$\R$-mixed-Hodge structure} on an $\R$-vector space $V$ is a pair $(W_{\ast},F^{\ast})$
such that:
\begin{enumerate}
\item $W_{\ast}$ is an increasing filtration which is bounded below,
\item $F^{\ast}$ is a decreasing filtration  on $V_{\C}$ called such that
the filtration on $Gr_{n}^{W} V_{\C}$ induced by $F^{\ast}$ is an $\R$-Hodge structure of weight $n$.
\end{enumerate}
We call $W_{\ast}$ the weight filtration and $F^{\ast}$ the Hodge filtration.
\end{definition}
\begin{proposition}\label{BIGG}{\rm (\cite[Proposition 1.9]{Mor})}
Let $(W_{\ast},F^{\ast})$ be an $\R$-mixed-Hodge structure on an $\R$-vector space $V$.
Define $V_{p,q}=R_{p,q}\cap L_{p,q}$ where
$R_{p,q}=W_{p+q}(V_{\C})\cap F^{p}(V_{\C})$ and
$L_{p,q}=W_{p+q}(V_{\C})\cap \overline{F^{q}(V_{\C})}+\sum_{i\ge 2} W_{p+q-i}(V_{\C})\cap \overline{F^{q-i+1}(V_{\C})}$.
Then we have 
  the   bigrading $V_{\C}=\bigoplus V_{p,q}$ such that:
$W_{i}(V_{\C})=\bigoplus_{p+q\le i}V_{p,q}$ and
$F^{i}(V_{\C})=\bigoplus_{p\ge i} V_{p,q}$.

\end{proposition}

\begin{proposition}{\rm (\cite[Proposition 1.11]{Mor})}
Let $V$ be an $\R$-vector space.
We suppose that we have a bigrading $V_{\C}=\bigoplus V^{p,q}$ 
such that 
  $\overline{V_{p,q}}=V_{q,p}$ modulo $\bigoplus_{r+s<p+q} V_{r,s}$ and the grading ${\rm Tot}^{\ast} V_{\ast,\ast}$ is bounded below where ${\rm Tot}^{r} V_{\ast,\ast}=\bigoplus_{p+q=r}V_{p,q}$.
Then the filtrations $W$ and $F$ such that
$ W_{i}(V_{\C})=\bigoplus_{p+q\le i}V_{p,q}$ and 
$F^{i}(V_{\C})=\bigoplus_{p\ge i} V_{p,q}$
give an $\R$-mixed Hodge structure on $V$.

\end{proposition}
\section{Sullivan's minimal models}

Let $k$ be a field of characteristic $0$.
We recall Sullivan's minimal model of a  differential graded algebra (shortly DGA).

%\begin{definition} A {\em differential graded algebra} (called  {\em DGA}) is a graded $k $-algebra $A^{\ast}$  with the following properties: \\(1)
%$A^{\ast}$ is graded commutative, i.e.
%\[y\wedge x=(-1)^{p\cdot q}x\wedge y \ \ \ x\in A^{p} \ \ \  y\in A^{q}.
%\]
%(2)
%There is a differential operator $d:A\rightarrow A$ of degree one such that $d\circ d=0$ and
%\[d(x\wedge y)=dx\wedge y+(-1)^{p}x\wedge dy  \ \ \  x\in A^{p}.
%\]
%\end{definition}

\begin{definition}
Let $A^{\ast}$ and $B^{\ast}$ be $k$-DGAs.
\begin{itemize}
\item A morphism $\phi:A^{\ast}\to B^{\ast}$ is a quasi-isomorphism (resp. $1$-quasi-isomorphism) if $\phi$ induces a cohomology isomorphism (resp.  isomorphisms on the $0$-th and first cohomologies and an injection on the second cohomology).

\item
$A^{\ast}$ and $B^{\ast}$ are {\em quasi-isomorphic} (resp. $1$-{\em quasi-isomorphic}) if there is a finite diagram of $k$-DGAs 
\[A\leftarrow C_{1}\rightarrow C_{2}\leftarrow\cdot \cdot \cdot \leftarrow C_{n}\rightarrow B
\]
such that all the morphisms are quasi-isomorphisms (resp. $1$-quasi-isomorphisms).

\end{itemize}
\end{definition}

\begin{definition}
A $k$-DGA  $A^{\ast}$  with a differential $d$  is {\em minimal}  if the following conditions hold.
\begin{itemize}
\item $A^{\ast}$ is a free graded algebra $\bigwedge \mathcal V^{\ast}$ generated by a graded $k$-vector space $\mathcal V^{\ast}$ with $\ast\ge 1$.

\item $\mathcal V^{\ast}$ admits a basis $\{x_{i}\}$ which is indexed by a well-ordered set such that $d x_{i}\in \bigwedge \langle \{x_{j} \vert j<i\}\rangle$

\item $d(\mathcal V^{\ast})\subset \bigwedge ^{2} \mathcal V^{\ast}$.

\end{itemize}
If a minimal  $k$-DGA $A^{\ast}$ is $1$-{\em minimal} if  $A^{\ast}$ is  generated by elements of degree $1$.
\end{definition}

%The DGA $\bigwedge \frak n^{\ast}$ of a nilpotent Lie algebra $\frak n$ is a minimal DGA.

\begin{definition}
Let $A^{\ast}$ be a $k$-DGA with $H^{0}(A^{\ast})=k$.
A minimal (resp. $1$-minimal)  $k$-DGA $\mathcal M$ is the {\em minimal model} (resp. $1$-{\em minimal model}) of $A^{\ast}$ if there is a quasi-isomorphism (resp. $1$-quasi-isomorphism) $\mathcal M\to A^{\ast}$.

\end{definition}

\begin{theorem}[\cite{Sul}]
For a $k$-DGA $A^{\ast}$ with $H^{0}(A^{\ast})=k$, 
the minimal model (resp. $1$-minimal model) of $A^{\ast}$ exists and it is unique up to $k$-DGA isomorphism.
\end{theorem}
Consequently,  if $A$ is quasi-isomorphic ($1$-quasi-isomorphic) to a $k$-DGA $B$, then $A^{\ast}$ and $B^{\ast}$ have the same minimal model (resp. $1$-minimal model).

\begin{definition}\label{formal}
Let $A^{\ast}$ be a $k$-DGA.
We say that $A^{\ast}$ is {\em formal} (resp. $1$-{\em formal}) if  $A^{\ast}$ is quasi-isomorphic (resp. $1$-quasi-isomorphic) to the $k$-DGA $H^{\ast}(A^{\ast})$ with the trivial differential.
\end{definition}

\begin{theorem}[\cite{DGMS}]
Let $M$ be a compact K\"ahler manifold.
Then the de Rham complex $A^{\ast}(M)$ is formal.
\end{theorem}

\begin{example}\label{nilmf}
Let $N$ be a real   simply connected nilpotent Lie group with a lattice (i. e. cocompact discrete subgroup) $\Gamma$.
Then the compact quotient $\Gamma\backslash N$ is called nilmanifold.
Let $\frak n$ be the Lie algebra of $N$.
Consider  the cochain  complex  $\bigwedge {\frak n}^{\ast}$ with the differential $d$ which is the dual to the Lie bracket.
Then, by the nilpotency, we can easily check that  $\bigwedge {\frak n}^{\ast}$ is a minimal DGA.
We regard $\bigwedge {\frak n}^{\ast}$ as the space of left-invariant differential forms on  $\Gamma\backslash N$.
Let $A^{\ast}(\Gamma\backslash N)$ be the de Rham complex of $\Gamma\backslash N$.
In \cite{Nom}, Nomizu proved that the canonical inclusion $\bigwedge {\frak n}^{\ast}\hookrightarrow A^{\ast}(\Gamma\backslash N)$ induces a cohomology isomorphism.
Thus the DGA $\bigwedge {\frak n}^{\ast}$ is the minimal model of  the de Rham complex $A^{\ast}(\Gamma\backslash N)$.
Since $\bigwedge {\frak n}^{\ast}$  is generated by elements of degree $1$, $\bigwedge {\frak n}^{\ast}$ is $1$-minimal and hence it is also the $1$-minimal model of  $A^{\ast}(\Gamma\backslash N)$.
In \cite{H}, the DGA $\bigwedge {\frak n}^{\ast}$ is formal if and only if the Lie algebra $\frak n$ is Abelian.
Hence, $A^{\ast}(\Gamma\backslash N)$ is formal  if and only if the nilmanifold  $\Gamma\backslash N$ is a torus.
Consequently, K\"ahler nilmanifolds are only tori.
\end{example}

\begin{remark}
In \cite{FM}, Fern\'andez and  Mu\~noz define the "$s$-formality" (\cite[Definition 2.4]{FM}).
We notice that the  $1$-formality in the sense of Fern\'andez-Mu\~noz is essentially different from the $1$-formaity as in Definition \ref{formal}.
In fact, for a nilpotent Lie algebra $\g$, the DGA $\bigwedge \g^{\ast}$ is $1$-formal  in the sense of Fern\'andez-Mu\~noz if and only if $\g$ is Abelian.
On the other hand, there exist non-Abelian nilpotent Lie algebras $\g$ such that  $\bigwedge \g^{\ast}$ is $1$-formal as in Definition \ref{formal}.
\end{remark}

\section{$1$-minimal model and $1$-formality}\label{1msul}
%\begin{definition}
%Let $A^{\ast}$ be a DGA with $H^{0}(A^{\ast})=k$.
%A minimal DGA $\mathcal M$ is the $1$-minimal model of $A^{\ast}$ if $\mathcal M$ is generated by elements of degree $1$ and there is a DGA homomorphism $\mathcal M\to A^{\ast}$ which induces  isomorphisms on the $0$-th and first cohomologies and an injection on the second cohomology.

%\end{definition}

%\begin{theorem}[\cite{Sul}]
%For a DGA $A^{\ast}$ with $H^{0}(A^{\ast})=k$, the $1$-minimal model of $A^{\ast}$ exists and it is unique up to DGA isomorphism.
%\end{theorem}

We  see the construction of $1$-minimal model of the de Rham complex in detail.
The article \cite[Chapter 3]{ABC} is a good reference.
Let $A^{\ast}$ be a $k$-DGA
 and $\mathcal M$  the $1$-minimal model of $A^{\ast}$  with a $1$-quasi-isomorphism $\phi:\mathcal M^{\ast}\to A^{\ast}$.
Then we have the canonical sequence of DGAs
\[{\mathcal M}^{\ast}(1)\subset {\mathcal M}^{\ast}(2)\subset\dots 
\]
such that:
\begin{enumerate}
\item $\mathcal M^{\ast}=\bigcup_{k=1}^{\infty} \mathcal M^{\ast}(k)$.
\item ${\mathcal M}^{\ast}(1)=\bigwedge {\mathcal V}_{1}$ with the trivial differential such that $\phi$ induces an isomorphism 
${\mathcal V}_{1}\to H^{1}(A^{\ast})$.
\item Consider the map $\phi_{1}:   {\mathcal V}_{1}\wedge {\mathcal V}_{1}\to H^{2}(A^{\ast})$ induced by $\phi$.
We have
${\mathcal M}^{\ast}(2)={\mathcal M}^{\ast}(1)\otimes \bigwedge {\mathcal V}_{2}$ such that the differential $d$ induces an isomorphism ${\mathcal V}_{2}\to {\rm Ker}\, \phi_{1}$.
\item Consider the map $\phi_{n}:   H^{2}({\mathcal M}^{\ast}(n))\to H^{2}(A^{\ast})$ induced by $\phi$ for each integer $n$.
We have
${\mathcal M}^{\ast}(n+1)={\mathcal M}^{\ast}(n)\otimes \bigwedge {\mathcal V}_{n+1}$ such that the differential $d:{\mathcal V}_{n+1}\to {\mathcal M}^{2}(n)$ induces an isomorphism ${\mathcal V}_{n+1}\to {\rm Ker}\, \phi_{n}$.

\end{enumerate}

 Let $M$  be a manifold and $A^{\ast}(M)$ the de Rham complex.
Suppose $A^{\ast}=A^{\ast}(M)$.
Dualizing  the sequence
\[{\mathcal M}^{\ast}(1)\subset {\mathcal M}^{\ast}(2)\subset \dots, 
\]
we obtain the tower of real nilpotent Lie algebras
\[\dots \to L_{2}\to L_{1}\to 0.
\]
Sullivan showed that this tower gives the Malcev Lie algebra of the fundamental group $\pi_{1}(M)$ (see \cite{DGMS}).

\begin{proposition}{\rm(\cite[Chapter 3]{ABC})}\label{quadr}
Let $M$ be a manifold.

Then the following conditions are equivalent:
\begin{enumerate}
 \item $A^{\ast}(M)$  is $1$-formal.

\item  The Malcev Lie algebra of the fundamental group $\pi_{1}(M)$  admits a quadratic presentation.

\item Consider the $1$-minimal model ${\mathcal M}^{\ast}$ of $A^{\ast}(M)$ and  the canonical sequence of DGAs
\[{\mathcal M}^{\ast}(1)\subset {\mathcal M}^{\ast}(2)\subset\dots 
\]
as above.
Then the map $H^{2}({\mathcal M}^{\ast}(1))\to H^{2}({\mathcal M}^{\ast})$ is surjective.
\end{enumerate}
\end{proposition}

By the result in \cite{DGMS}, for a compact K\"ahler manifold $M$, the Malcev Lie algebra of  the fundamental group $\pi_{1}(M)$ admits a quadratic presentation.
For example, for a compact surface $S$ with the genus $g\ge 2$,  the Malcev Lie algebra of  the fundamental group $\pi_{1}(S)$  is the free Lie algebra on $X_{1},\dots X_{g}, Y_{1},\dots, Y_{g}$ modulo the relation $\sum_{i=1}^{g}[X_{i},Y_{i}]=0$ (see \cite[Section 12]{Sul}).

\begin{example}\label{HeiHei}
Let $N$ be a real   simply connected nilpotent Lie group with a lattice $\Gamma$ and $\frak n$ be the Lie algebra of $N$.
Then, as in Example \ref{nilmf}, the DGA $\bigwedge \frak n^{\ast}$ is the  $1$-minimal model of  $A^{\ast}(\Gamma\backslash N)$.
Since the fundamental group of the nilmanifold $\Gamma\backslash N$ is $\Gamma$, the Malcev Lie algebra of $\Gamma$ is $\frak n$.
Suppose $N=H_{2n+1}$ where  $H_{2n+1}$ is the $2n+1$-dimensional real  Heisenberg group.
Then the DGA $\bigwedge \frak n^{\ast}$ is given by $ \frak n^{\ast}=\langle x_{1},\dots ,x_{2n}, y\rangle$ so that $dx_{i}=0$ and 
\[dy=x_{1}\wedge x_{2}+\dots + x_{2n-1}\wedge x_{2n}.
\]
As in the above argument, we have ${\mathcal V}_{1}=\langle x_{1},\dots x_{2n}\rangle$ and  ${\mathcal V}_{2}=\langle y\rangle$.
Suppose  $n\ge 2$.
Then we have $H^{2}(\frak n)={\mathcal V}_{1}\wedge {\mathcal V}_{1}/ \langle dy\rangle$ and hence the map $H^{2}({\mathcal M}^{\ast}(1))\to H^{2}({\mathcal M}^{\ast})$ is surjective.
By Proposition \ref{quadr}, the Malcev Lie algebra of $\Gamma$ admits a quadratic presentation.
On the other hand, if $n=1$, then we can check that $\Gamma $ does not admit a quadratic presentation.
Consequently, a lattice in  the $3$-dimensional real  Heisenberg group can not be the fundamental  group  of any compact K\"ahler manifold.
\end{example}

See \cite{CT} for more examples of nilpotent groups admitting quadratic presentations.

\section{Morgan's mixed Hodge diagrams}
\begin{definition}{\rm (\cite[Definition 3.5]{Mor})}
An $\R$-{\em mixed-Hodge diagram} is a pair of filtered $\R$-DGA $(A^{\ast}, W_{\ast})$ and bifiltered $\C$-DGA $(E^{\ast}, W_{\ast},F^{\ast})$ and filtered DGA map $\phi:(A^{\ast}_{\C},W_{\ast})\to (E^{\ast},W_{\ast})$  such that:
\begin{enumerate}
\item $\phi$  induces an isomorphism  $\phi^{\ast}:\,_{W}E^{\ast,\ast}_{1}(A^{\ast}_{\C})\to \,_{W}E^{\ast,\ast}_{1}(E^{\ast})$ where $ \,_{W}E_{\ast}^{\ast,\ast}(\cdot)$ is the spectral sequence for the decreasing filtration $W^{\ast}=W_{-\ast}$.
\item The differential $d_{0}$ on $\,_{W}E^{\ast,\ast}_{0}(E^{\ast})$ is strictly compatible with the filtration induced by $F$.
\item The filtration on $\,_{W}E_{1}^{p,q}(E^{\ast})$ induced by $F$ is a $\R$-Hodge structure of weight $q$ on $\phi^{\ast}(\,_{W}E^{\ast,\ast}_{1}(A^{\ast}))$.

\end{enumerate}
\end{definition}

\begin{theorem}\label{midimi}{\rm (\cite[Theorem 4.3]{Mor})}
Let $\{(A^{\ast}, W_{\ast}), (E^{\ast}, W_{\ast},F^{\ast}),\phi\}$ be an $\R$-mixed-Hodge diagram.
Define the filtration $W^{\prime}_{\ast}$ on $H^{r}(A^{\ast})$ (resp $H^{r}(E^{\ast}))$
as $W^{\prime}_{i}H^{r}(A^{\ast})=W_{i-r}(H^{r}(A^{\ast}))$ (resp. $W^{\prime}_{i}H^{r}(E^{\ast})=W_{i-r}(H^{r}(E^{\ast}))$).
Then the filtrations $W^{\prime}_{\ast}$ and $F^{\ast}$ on $H^{r}(E^{\ast})$ give an $\R$-mixed-Hodge  on 
$\phi^{\ast}(H^{r}(A^{\ast}))$.
\end{theorem}

%\begin{theorem}[\cite{Mor}]\label{MMM}
%Let $\{(A, W), (E, W,F),\phi\}$ be an $\R$-mixed-Hodge diagram.
%Then the minimal model $\mathcal M^{\ast}$  of the DGA $E$
%satisfies the following conditions:
%\begin{itemize}
%\item $\mathcal M^{\ast}$ admits a bigrading 
%\[\mathcal M^{\ast}=\bigoplus_{p,q\ge0}\mathcal M^{\ast}_{p,q}
%\]
%such that $\mathcal M^{\ast}_{0,0}=\mathcal M^{0}=\C$ and 
%the product and the differential are of type $(0,0)$.

%\item The bigrading $\bigoplus_{p,q\ge0}\mathcal M^{\ast}_{p,q}$ induces an $\R$-mixed-Hodge structure on  the minimal model of $A$.

%\item The bigrading  on the cohomology $H^{\ast}(E)\cong  H^{\ast}(\mathcal M)$ induced by the bigrading on $\mathcal M^{\ast}$ agrees  with the bigrading induced by the $\R$-mixed-Hodge structure (see Proposition \ref{BIGG}) as in Theorem \ref{midimi}.

%\end{itemize}
%\end{theorem}

\begin{theorem}{\rm (\cite[Section 6]{Mor})]}\label{MMM11}
Let $\{(A^{\ast}, W_{\ast}), (E^{\ast}, W_{\ast},F^{\ast}),\phi\}$ be an $\R$-mixed-Hodge diagram.
Then the minimal model (resp.  $1$-minimal model) $\mathcal M^{\ast}$  of the DGA $E^{\ast}$ with a quasi-isomorphism (resp. $1$-quasi-isomorphism) $\phi:\mathcal M^{\ast}\to E^{\ast}$ 
satisfies the following conditions:
\begin{itemize}
\item $\mathcal M^{\ast}$ admits a bigrading 
\[\mathcal M^{\ast}=\bigoplus_{p,q\ge0}\mathcal M^{\ast}_{p,q}
\]
such that $\mathcal M^{\ast}_{0,0}=\mathcal M^{0}=\C$ and 
the product and the differential are of type $(0,0)$.

%\item The bigrading $\bigoplus_{p,q\ge0}\mathcal M^{\ast}_{p,q}$ induces an $\R$-mixed-Hodge structure on  the $1$-minimal model of $A$.

\item 
Consider the bigrading $H^{\ast}(E^{\ast})=\bigoplus V_{p,q}$ for the  $\R$-mixed-Hodge structure  as in Theorem \ref{midimi}.
Then  $\phi^{\ast}:H^{\ast}(\mathcal M^{\ast})\to H^{\ast}(E^{\ast})$ 
sends $H^{\ast}(\mathcal M^{\ast}_{p,q})$ to $V_{p,q}$.
\end{itemize}
\end{theorem}
We explain Morgan's construction for $1$-minimal models.
For an  $\R$-mixed-Hodge diagram $\{(A^{\ast}, W_{\ast}), (E^{\ast}, W_{\ast},F^{\ast}),\phi\}$,   let $\mathcal M^{\ast}$ be  the $1$-minimal model of $E^{\ast}$ with a $1$-quasi-isomorphism $\phi:\mathcal M^{\ast}\to E^{\ast}$.
Take the canonical sequence of DGAs
\[{\mathcal M}^{\ast}(1)\subset {\mathcal M}^{\ast}(2)\subset\dots 
\]
as in Section \ref{1msul}.
Then the bigrading  as in Theorem \ref{MMM11} constructed by the following inductive way.
 For each $n$,  we let $\mathcal M^{\ast}(n)$ have a bigrading
\[\mathcal M^{\ast}(n)=\bigoplus_{p,q\ge0}\mathcal M^{\ast}_{p,q}(n)
\]
such that:
\begin{itemize}
\item  $\mathcal M^{\ast}_{0,0}(n)=\mathcal M^{0}(n)=\C$ and 
the product and the differential are of type $(0,0)$.

%\item The bigrading $\bigoplus_{p,q\ge0}\mathcal M^{\ast}_{p,q}$ induces an $\R$-mixed-Hodge structure on  the $1$-minimal model of $A$.

\item 
Consider the bigrading $H^{\ast}(E^{\ast})=\bigoplus V_{p,q}$ for the  $\R$-mixed-Hodge structure  as in Theorem \ref{midimi}.
Then  $\phi^{\ast}:H^{\ast}(\mathcal M^{\ast}(n))\to H^{\ast}(E^{\ast})$ 
sends $H^{\ast}(\mathcal M^{\ast}_{p,q}(n))$ to $V_{p,q}$.

\item
For ${\mathcal M}^{\ast}(n+1)={\mathcal M}^{\ast}(n)\otimes \bigwedge {\mathcal V}_{n+1}$,
${\mathcal V}_{n+1}$ has a bigrading   such that  the differential $d:{\mathcal V}_{n+1}\to {\mathcal M}^{2}(n)$ is  compatible with bigradings and
   the bigrading on ${\mathcal M}^{\ast}(n+1)$  is the multiplicative extension of the bigradings on ${\mathcal V}_{n+1}$ and ${\mathcal M}^{\ast}(n)$.

\end{itemize}

\section{Mixed Hodge diagrams of Sasakian manifolds}

%\subsection{Models of Sasakian manifolds}
Let $M$ be a compact $(2n+1)$-dimensional Sasakian manifold with a Sasakian metric $g$ and 
 $\eta$ the contact structure associated with the Sasakian structure.
Take $\xi$ the Reeb vector field.
$\xi$ gives the $1$-dimensional foliation $\mathcal F_{\xi}$.
Let $A^{\ast}(M)$ be the de Rham complex of $M$.
A differential form $\alpha\in A^{\ast}(M)$ is basic if  $\iota_{\xi}\alpha=0$ and $\iota_{\xi}d\alpha=0$.
Denote by $A_{B}^{\ast}(M)$ the differential graded algebra  of the basic differential forms on $M$ and  denote by $H^{\ast}_{B}(M,\R)$ (resp. $H^{\ast}_{B}(M,\C)$) the cohomology of $A_{B}^{\ast}(M)$ (resp. $A_{B}^{\ast}(M)\otimes\C$).
For a closed basic form $\alpha$, we denote by $[\alpha]_{B}$ the the cohomology class of $\alpha$ in  $H^{\ast}_{B}(M,\R)$ (resp. $H^{\ast}_{B}(M,\C)$).

We have the transverse complex structure on ${\rm Ker}\, \iota_{\xi}\subset \bigwedge TM^{\ast}$ and we obtain the bigrading $A^{r}_{B}(M)\otimes \C=\bigoplus_{p+q=r} A^{p,q}_{B}(M)$ with the bidifferential $d=\partial_{B}+\bar\partial_{B}$.
We have the $\partial_{B}\bar\partial_{B}$-Lemma:
\begin{proposition}\label{del}{\rm (\cite[Proposition 3.7]{Ti})}
For each $r$, in $A^{r}_{B}(M)\otimes \C$, we have
\[{\rm Ker}\, \partial_{B}\cap  {\rm Ker}\, \bar\partial_{B}\cap {\rm Im}\, d={\rm Im}\, \partial_{B}\bar\partial_{B}.
\]
\end{proposition}

We define the DGA $A^{\ast}=H^{\ast}_{B}(M,\R)\otimes \bigwedge \langle y \rangle$ such that $y$ is an element of degree $1$ and $dy=[d\eta]_{B}$.
By using Proposition \ref{del}, we obtain the following theorem.
\begin{theorem}\label{for}{\rm (\cite[Section 4.3]{Ti})}
The DGAs $A^{\ast}$ and $A^{\ast}(M)$ are quasi-isomorphic.
\end{theorem}

%\subsection{Mixed Hodge diagrams of Sasakian manifolds}
Define the basic Bott-Chern cohomology $H^{\ast,\ast}_{B}(M)$ as
\[H^{\ast,\ast}_{B}(M)=\frac{{\rm Ker}\,\partial_{B}\cap {\rm Ker}\,\bar\partial_{B}}{{\rm Im}\,\partial_{B}\bar\partial_{B}}.
\]
Then we have $\overline{H^{p,q}_{B}(M)}=H^{q,p}_{B}(M)$ and 
the natural map
\[{\rm Tot} ^{\ast}H^{\ast,\ast}_{B}(M)\to H^{\ast}_{B}(M,\C).
\]
By  $\partial_{B}\bar\partial_{B}$-Lemma, this natural map is an isomorphism (see \cite[Remark 5.16]{DGMS}).
Hence we have the $\R$-Hodge structure of weight $r$ on $H^{r}_{B}(M,\R)$ as
\[H^{r}_{B}(M,\C)=\bigoplus_{p+q=r} H^{p,q}_{B}(M).
\]
We consider the
bigrading $A^{\ast}_{\C}=\bigoplus A^{\ast}_{p,q}$ such that
\[A^{\ast}_{p,q}=H^{p,q}_{B}(M)\oplus H^{p-1,q-1}_{B}(M)\wedge y.
\]

\begin{definition}
\begin{enumerate}
\item
We define the increasing filtration  $W_{\ast}$ on the DGA $A^{\ast}$
such that 
$W_{-1}A^{\ast}=0$,
$W_{0}A^{\ast}=H^{\ast}_{B}(M,\R)
$
and
$W_{1}A^{\ast}=A^{\ast}$.

\item
We define the decreasing filtration  $F^{\ast}$ on the DGA $A^{\ast}_{\C}$ such that
\[F^{k} (A^{\ast}_{\C})=\bigoplus_{p\ge k} A^{\ast}_{p,q}.
\]

\end{enumerate}
\end{definition}

\begin{theorem}\label{SASAH}
The pair $(A^{\ast},W_{\ast})$, $(A^{\ast}_{\C},W_{\ast},F^{\ast})$ and ${\rm id}:A^{\ast}_{\C}\to A^{\ast}_{\C}$ is a $\R$-mixed-Hodge diagram.
\end{theorem}
\begin{proof}
We can easily check that the differential $d_{0}$ on  $\,_{W}E_{0}^{p,q}(A^{\ast}_{\C})$ is zero and we obtain
 \[\,_{W}E_{1}^{0,q}(A^{\ast}_{\C})\cong H^{q}_{B}(M,\C),\]
 \[\,_{W}E_{1}^{-1,q}(A^{\ast}_{\C})\cong H^{q-2}_{B}(M,\C)\wedge y\]
and other $\,_{W}E_{1}^{p,q}(A^{\ast}_{\C})$ is trivial.
Hence, by $H^{r}_{B}(M,\C)=\bigoplus_{p+q=r} H^{p,q}_{B}(M)$ and $\overline{H^{p,q}_{B}(M)}=H^{q,p}_{B}(M)$,
the filtration $F$ induces $\R$-Hodge structures of weight $q$ on  $\,_{W}E_{1}^{0,q}(A^{\ast}_{\C})$ and $\,_{W}E_{1}^{-1,q}(A^{\ast}_{\C})$
such that $\,_{W}E_{1}^{-1,q}(A^{\ast}_{\C})=\bigoplus_{s+t=q} V_{s,t}$ with $V_{s,t}=H^{s-1,t-1}_{B}(M)\wedge y$.
Hence the theorem follows.
\end{proof}

Hence, by Theorem \ref{MMM11}, we obtain the following theorem.

\begin{theorem}\label{SasMMM11}
Let $M$ be a $2n+1$-dimensional compact Sasakian manifold and $A^{\ast}_{\C}(M)$ the de Rham complex of $M$.
Consider the minimal model $\mathcal M$ (resp $1$-minimal model) of $A^{\ast}_{\C}(M)$ with a quasi-isomorphism  (resp. $1$-quasi-isomorphism) $\phi:\mathcal M\to A_{\C}^{\ast}(M)$.
Then we have:
\begin{enumerate}
\item The real de Rham cohomology $H^{\ast}(M,\R)$ admits a 
$\R$-mixed-Hodge structure.

\item 
$\mathcal M^{\ast}$ admits a bigrading 
\[\mathcal M^{\ast}=\bigoplus_{p,q\ge0}\mathcal M^{\ast}_{p,q}
\]
such that $\mathcal M^{\ast}_{0,0}=\mathcal M^{0}=\C$ and 
the product and the differential are of type $(0,0)$.

\item
Consider the bigrading $H^{\ast}(M,\C)=\bigoplus V_{p,q}$ for the  $\R$-mixed-Hodge structure.
Then the induced map $\phi^{\ast}:H^{\ast}(\mathcal M^{\ast})\to H^{\ast}(M,\C)$ 
sends $H^{\ast}(\mathcal M^{\ast}_{p,q})$ to $V_{p,q}$.

\end{enumerate}
\end{theorem}

 By this theorem, we have a canonical mixed Hodge structure on the cohomology $H^{\ast}(M,\R)$.
In more detail the mixed Hodge structure  is given by the bigrading $A^{\ast}_{\C}=\bigoplus A^{\ast}_{p,q}$ such that
\[A^{\ast}_{p,q}=H^{p,q}_{B}(M)\oplus H^{p-1,q-1}_{B}(M)\wedge y.
\]
We can easily show the following proposition.
\begin{proposition}\label{Hods}
Consider the bigrading $H^{\ast}(M,\C)=\bigoplus V_{p,q}$ as in  Proposition \ref{BIGG}.
\begin{enumerate}
\item $H^{1}(M,\C)=V_{1,0}\oplus V_{0,1}$.

\item 
$V_{n+1,n+1}=H^{2n+1}(M,\C)$.

\item 
If $n\ge 2$, then $H^{2}(M,\C)=V_{2,0}\oplus V_{1,1}\oplus V_{0,2}$.
\end{enumerate}
\end{proposition}
\begin{proof}
The first and second assertions are easy.

Suppose $n\ge 2$.
Then it is known that the map $[d\eta]\wedge :H^{1}_{B}(M,\C)\to H^{3}_{B}(M,\C)$ is injective (see \cite[Section 7.2]{BoG}).
Hence we have ${\rm Ker}\, d_{\vert A_{\C}^{2}} \subset H^{2}_{B}(M,\C)$.
Thus the third assertion follows. 
\end{proof}

\begin{example}
Let $N=H_{2n+1}$ and $\Gamma$ be a lattice in $N$.
Consider the DGA $\bigwedge \frak n^{\ast}=\bigwedge \langle x_{1},\dots ,x_{2n}, y\rangle$ as in Example \ref{HeiHei}.
Then the nilmanifold $\Gamma\backslash N$ is a Sasakian manifold with the left-invariant contact structure $y$.
In this case, the foliation induced by  Reeb vector field is a principal $S^{1}$ bundle over the $n$-dimensional Abelian variety $T$ 
associated with the K\"ahler  class on $T$.
We have $H^{\ast}_{B}(M,\R) \cong H^{\ast}(T,\R)\cong \bigwedge \langle x_{1},\dots ,x_{2n}\rangle$.
Thus the DGA $A^{\ast}=H^{\ast}_{B}(M,\R)\otimes \bigwedge \langle y \rangle$ is identified with $\bigwedge \frak n^{\ast}$ and hence it is the minimal model of the de Rham complex $A^{\ast}(\Gamma\backslash N)$.
For ${\mathcal V}= \langle x_{1},\dots ,x_{2n}\rangle$,
take ${\mathcal V}\otimes \C=\mathcal V^{1,0}\oplus \mathcal V^{0,1}$ associated with the complex structure on $T$.
Take  the bigrading on $\bigwedge \frak n^{\ast}_{\C}=\bigwedge (\mathcal V^{1,0}\oplus \mathcal V^{0,1})\otimes \bigwedge \langle y \rangle$ so that $y$ is considered as an element of type $(1,1)$.
Then we can check that this bigrading   is  a bigrading    as in Theorem \ref{SasMMM11} (see the explanation after Theorem \ref{MMM11}).
As in Proposition \ref{Hods}, if  $n\ge 2$ we can easily check that $H^{2}(\Gamma\backslash N,\C)=V_{2,0}\oplus V_{1,1}\oplus V_{0,2}$.
But if $n=1$, we have  $H^{2}(\Gamma\backslash N,\C)=V_{2,1}\oplus V_{1,2}$
\end{example}

\section{$1$-formality of Sasakian manifolds}

%\begin{theorem}
%Let $\{(A, W), (E, W,F),\phi\}$ be an $\R$-mixed-Hodge diagram.
%Consider the bigrading $H^{\ast}(E)=\bigoplus V^{p,q}$ for the  $\R$-mixed-Hodge structure  as in Theorem \ref{midimi}.
%Suppose that  $H^{1}(E)=V^{1,0}\oplus V^{0,1}$ and 
% $H^{2}(E)=V^{2,0}\oplus V^{1,1}\oplus V^{0,2}$.
%Then $E$ is $1$-formal.
%\end{theorem}

\begin{theorem}
Let $M$ be a compact $2n+1$-dimensional Sasakian manifold with $n\ge 2$.
Then the DGA $A^{\ast}_{\C}(M)$ is $1$-formal equivalently the Malcev Lie algebra of $\pi_{1}(M)$ admits a quadratic presentation.
\end{theorem}
\begin{proof}
Consider the $1$-minimal model $\mathcal M^{\ast}$  of the de Rham complex $A^{\ast}_{\C}(M)$ with a $1$-quasi-isomorphism  $\phi:\mathcal M^{\ast}\to A^{\ast}_{\C}(M)$.
We take a bigrading 
\[\mathcal M^{\ast}=\bigoplus_{p,q\ge0}\mathcal M^{\ast}_{p,q}
\]
as Theorem \ref{SasMMM11}.
Since the induced injection  $\phi^{2}:H^{\ast}(\mathcal M^{\ast})\to H^{2}(M,\C)$ 
sends $H^{2}(\mathcal M^{\ast}_{p,q})$ to $V_{p,q}$,
by Proposition \ref{Hods}, we have 
$H^{2}(\mathcal M^{\ast})=H^{2}(\mathcal M^{\ast}_{2,0}\oplus \mathcal M^{\ast}_{1,1}\oplus \mathcal M^{\ast}_{0,2})$.
We also consider the canonical sequence of DGAs
\[{\mathcal M}^{\ast}(1)\subset {\mathcal M}^{\ast}(2)\subset \dots .
\]
as Section \ref{1msul}.
By Proposition \ref{Hods}, ${\mathcal V}_{1}$ is spanned by elements of  type $(1,0)$ and $(0,1)$.
Since we take $d{\mathcal V}_{2}\subset {\mathcal V}_{1}\wedge {\mathcal V}_{1}$,
we can see that ${\mathcal V}_{2}$ is spanned by elements of  type $(2,0)$, $(1,1)$ and $(0,2)$.
For $n\ge 3$,  since we take $d{\mathcal V}_{n}\subset \bigoplus_{i,j<0,\, i+j>2}  {\mathcal V}_{i}\wedge {\mathcal V}_{j}$,
  ${\mathcal V}_{n}$  is spanned by  elements of type $(p,q)$ with $p+q\ge 3$.
Hence we have 
\[H^{2}(\mathcal M^{\ast})=H^{2}(\mathcal M^{\ast}_{2,0}\oplus \mathcal M^{\ast}_{1,1}\oplus \mathcal M^{\ast}_{0,2})= {\mathcal V}_{1}\wedge {\mathcal V}_{1}/d({\mathcal V}_{2}) .\]
Hence the theorem follows from Proposition \ref{quadr}.

\end{proof}

Let $H_{3}$ be the $3$-dimensional real Heisenberg group and $\Gamma$ a lattice in $H_{3}$.
Then the nilmanifold $\Gamma\backslash H_{3}$ is Sasakian.
In \cite{Ti}, Tievsky observed that for any compact even dimensional simply connected manifold $M$, the product $\Gamma\backslash H_{3}\times M$ is not Sasakian.
We can extend this observation.

\begin{proposition}
For any compact even dimensional manifold $M$ (not necessarily simply connected), 
the product $\Gamma\backslash H_{3}\times M$ is not Sasakian.
\end{proposition}
\begin{proof}
 The nilmanifold $\Gamma\backslash H_{3}$ has a non-trivial Massey triple product on the first cohomology and it is an obstruction of $1$-formality.
The product $\Gamma\backslash H_{3}\times M$ also has a non-trivial Massey triple product on the first cohomology.
Hence $\Gamma\backslash H_{3}$ is not Sasakian.
\end{proof}

It is known that for a compact contact manifold $X$ and compact surface $S$,
the product $X\times S$ admits a contact structure (see \cite{Bou}, \cite{Bow}).
Hence we can construct many non-Sasakian contact manifolds.
\begin{corollary}
Let $S_{1},\dots, S_{n}$ be compact surfaces.
Then the product $\Gamma\backslash H_{3}\times S_{1}\times \dots \times S_{n}$ is a non-Sasakian contact manifold.
\end{corollary}

\section{Sasakian nilmanifolds}\label{nini}
First, we show the following proposition inspired by Hain's paper \cite{Hali}.
\begin{proposition}\label{nilll}
Let $\frak n$ be a real $2n+1$-dimensional nilpotent Lie algebra.
Suppose that the DGA $\bigwedge \frak n^{\ast}_{\C}$ admits a bigrading $\bigwedge \frak n^{\ast}_{\C}=\bigoplus \mathcal M_{p,q}^{\ast}$ such that 
$\mathcal M^{\ast}_{0,0}=\bigwedge^{0} \frak n^{\ast}_{\C}=\C$ and 
the product and the differential are of type $(0,0)$.
We suppose that $H^{1}(\frak n,\C)={\rm Ker}\, d_{\vert \frak n^{\ast}_{\C}}={\mathcal M}^{\ast}_{1,0}\oplus {\mathcal M}^{\ast}_{0,1}$
and $H^{2n+1}(\frak n,\C)=\bigwedge^{2n+1} \frak n^{\ast}_{\C}={\mathcal M}_{n+1,n+1}^{\ast}$.
Then we have $\dim H^{1}(\frak n,\R)=2n$.
\end{proposition}
\begin{proof}
Let $w_{r}=\sum_{p+q=r} \dim {\mathcal M}^{1}_{p,q}$.
Then, by the assumption, we have $\sum_{r} w_{r}=2n+1$ and $\sum_{r} rw_{r}=2n+2$.
Hence $\sum_{r} (r-1)w_{r}=1$ and this implies $w_{2}=1$, $w_{r}=0$ for $r\ge 3$ and so $w_{1}=2n$.
By the assumption, we have $w_{1}= \dim H^{1}(\frak n,\C)$.
Hence the proposition follows.
\end{proof}

\begin{theorem}[\cite{Nil}]
Let $N$ be a real  $(2n+1)$-dimensional simply connected nilpotent Lie group with a lattice $\Gamma$.
If the nilmanifold $\Gamma\backslash N$ admits a Sasakian structure, then $N$ is the real $(2n+1)$-dimensional  Heisenberg group.
\end{theorem}
\begin{proof}
Let $\frak n$ be the Lie algebra of $N$.
Then we can say that $\bigwedge \frak n^{\ast}$ is the minimal model of $A^{\ast}(\Gamma\backslash N)$ (see Example \ref{nilmf}).
By Theorem \ref{SasMMM11} and Proposition \ref{Hods},
the assumptions of Proposition \ref{nilll} hold and hence
 we can say  that $\bigwedge^{1} \frak n^{\ast}=V_{1}\oplus \langle v_{2}\rangle $ such that $H^{1}(\n,\R)=V_{1}$, $\dim V_{1}=2n$ and $dv_{2}\in \bigwedge^{2} V_{1}$.
By theorem \ref{for}, we have a morphism $\phi:\bigwedge \frak n^{\ast}\to A^{\ast}= H^{\ast}_{B}(M,\R)\otimes\bigwedge \langle y\rangle  $ which induces a cohomology isomorphism.
By $H^{1}(A^{\ast})=H^{1}_{B}(M,\R)$, we have $\phi(V_{1})=H^{1}_{B}(M,\R)$.
By $H^{2n+1}(A^{\ast})=H^{2n}_{B}(M,\R)\wedge y$, we have $\phi(\bigwedge^{2n+1} \frak n^{\ast})=H^{2n}_{B}(M,\R)\wedge y$.
Hence we can obtain $\phi(v_{2})=a+cy$ such that  $a\in H^{1}_{B}(M,\R)$ and $c\in \R$ with $c\not=0$.
This implies that $\phi(dv_{2}) \in \langle [d\eta]\rangle$.
It is known that $[d\eta]^{2n}\not=0$ (see \cite[Section 7.2]{BoG}).
Hence $dv_{2}$ is non-degenerate in $\bigwedge^{2} V_{1}$ and so the Theorem follows.

\end{proof}

%\section{Non-Sasakian Case}

\end{document}